\documentclass[reqno,12pt]{amsart}
\usepackage{amscd}

\usepackage{mathrsfs}
\usepackage{indentfirst,latexsym,bm,amsmath,amssymb,amsthm,amscd}
\usepackage{diagrams}
\newtheorem{theorem}{Theorem}[section]
\newtheorem{lemma}[theorem]{Lemma}

\newtheorem{corollary}[theorem]{Corollary}

\theoremstyle{definition}
\newtheorem{definition}[theorem]{Definition}

\newtheorem{remark}[theorem]{Remark}

\numberwithin{equation}{section}

 \DeclareMathOperator{\Div}{Div}

\begin{document}

\title
{Canonical maps of surfaces defined by Abelian covers}

\author[Rong Du]{Rong Du$^{\dag}$}
\address{Department of Mathematics\\
Shanghai Key Laboratory of PMMP\\
East China Normal University\\
Rm. 312, Math. Bldg, No. 500, Dongchuan Road\\
Shanghai, 200241, P. R. China} \email{rdu@math.ecnu.edu.cn}
\address{Current address: Department of Mathematics\\
The University of Hong Kong\\
Rm. 408, Run Run Shaw Bldg, Pokfulam, Hong Kong}
\email{rongdu@hku.hk}
\author[Yun Gao]{Yun Gao$^{\dag\dag}$}

\address{Department of Mathematics, Shanghai Jiao Tong University,
Shanghai 200240, P. R. of China}

\email{gaoyunmath@sjtu.edu.cn}

\thanks{$^{\dag}$ The Research Sponsored by The National Natural Science Foundation of China, Shanghai Pujiang Program and the Scientific Research Foundation for the Returned Overseas Chinese Scholars, State Education Ministry.}
\thanks{$^{\dag\dag}$ The Research Sponsored by National Natural Science Foundation of China and The Innovation Program of Shanghai
Municipal Education Commission.}

 \maketitle

 \begin{abstract}{In this paper, we classified the surfaces whose canonical maps are abelian covers over $\mathbb{P}^2$. Moveover, we give defining equations for Perssson's surface and Tan's surfaces with odd canonical degrees explicitly. }\end{abstract}

\section{Introduction}\label{secintro}
Let $X$ be a minimal surface of general type over $\mathbb{C}$. As
usual, let $p_g$, $q$, $\chi(\mathcal{O}_X)$, $K_X^2$ be the
numerical invariants of $X$,
$$\varphi_{K_X}: X\rightarrow \Sigma \subset\mathbb{P}^{p_g-1}$$
be the canonical map of $X$ and $d=$deg$\varphi_{K_X}$ be the
canonical degree of $X$. It is well known that the behavior of
$\varphi_{K_X}$ is quite complicated, however the degree of it can
not be too large. In 1979, Beauville (\cite{Bea}) proved that when
$\varphi_{K_X}$ is a generically finite map, the degree is at most
$36$. Furthermore, deg$\varphi_{K_X}$=36 if and only if $K_X^2=36$,
$p_g=3$, $q=0$, $\Sigma=\mathbb{P}^2$ and $|K_X|$ is base point
free. Later, Xiao also found some restrictions on surfaces with high
canonical degrees (\cite{Xiao}). For surfaces of general type, since
the canonical degree is bounded above, it is interesting to know
which positive integers $d$'s occur as the degree of the canonical
map. Among all known surfaces with highest canonical degree less
than $36$ is the surface with canonical degree $16$ which was
constructed by Persson(\cite{Per}) as follows. let $\pi:
X\rightarrow Y$ be a double cover of a Campedelli surface $Y$
branched along $2K_Y$, where the Campedelli surface is the surface
of general type with $p_g=q=0$, $K_Y^2=2$, and $|2K_Y|$ is base
point free. He found that $\varphi_{K_X}=\pi \circ \varphi_{K_Y}$,
hence $\varphi_{K_X}$ is degree $16$ over $\mathbb{P}^2$. We see
that his construction is based on the the existence of Campedelli
surface, however the construction of a Campedelli surface is not an
easy work (\cite{Pet}).  The remaining problem that whether there
exists a surface of general type over $\mathbb{P}^2$ with canonical
degree of the gap between $17$ and $35$ is still open.

If $\Sigma$ is a canonical surface, there are plenty of examples
(see \cite{Bea}, \cite{Cat1}, \cite{V-Z} ) with canonical degrees
being $2$. For $d=3$ and $d=5$, Tan (\cite{Tan}) and Pardini
(\cite{Par1}) constructed several surfaces independently. When
$p_g(\Sigma)=0$, Beauville (\cite{Bea}) constructed surfaces with
$\chi(\mathcal{O}_X)$ arbitrarily large and the canonical degrees
$2,4,6$ and $8$ . For $d=9$, Tan also constructed a surface in
\cite{Tan}. However, as far as we know that it is still not clear
which positive integers $d$'s occur as the degree of the canonical
map even when $\Sigma=\mathbb{P}^2$.

In this paper, we show that if the canonical map is an abelian cover
over $\mathbb{P}^2$ then the only possible canonical degrees of a
surface of general type are $2, 3, 4, 6, 8, 9, 16$ by explicit
constructions. Moreover, by using abelian cover, we list the
defining equations for Perssson's surface and Tan's surfaces with
odd canonical degrees.

We list our main theorems as follows.

\begin{theorem}
Assume $\varphi: X\longrightarrow\mathbb{P}^2$ is an abelian cover.
\begin{enumerate}
\item[(1)] Let $\deg\varphi=16$, then $\varphi=\varphi_{K_X}$ if and only if $X$
is defined by
\begin{equation*}
z_1^2=\ell_1\ell_3\ell_4\ell_7,\hskip0.3cm
z_2^2=\ell_1\ell_2\ell_4\ell_5,\hskip0.3cm
z_3^2=\ell_1\ell_2\ell_3\ell_6,\hskip0.3cm
z_4^2=\ell_2\ell_5\ell_6\ell_8.\label{8}
\end{equation*}
In particular, the surface defined by the first three equations is a
Campedelli surface.
\item[(2)]Let $\deg\varphi=9$, then $\varphi=\varphi_{K_X}$ if and only if
$X$ is defined by $z_1^3=a_1a_2^2,z_2^3=a_1a_2a_3$.
\item[(3)] Let $\deg\varphi=8$, then $\varphi=\varphi_{K_X}$ if and only if
$X$ is defined by either $z_1^2=\ell_1\ell_2\ell_7\ell_8,\,\,
z_2^2=\ell_3\ell_4\ell_7\ell_8,\,\,z_3^2=\ell_5\ell_6\ell_7\ell_8$,
or $z_1^2=a_1a_4,z_2^2=a_2a_4,z_3^2=a_3a_4$, or
$z_1^2=a_1a_2,z_2^4=a_2^3a_3$.
\item[(4)] Let $\deg\varphi=6$, then
$\varphi=\varphi_{K_X}$ if and only if $X$ is defined by
$z_1^2=a_1a_2,z_2^3=a_2a_3^2$.
\item[(5)] Let $\deg\varphi=4$, then
$\varphi=\varphi_{K_X}$ if and only if
 $X$ is defined by either
 $z^4=a_1^2b_2^3$,or $z_1^2=b_1$,$z_2^2=b_2$.
\item[(6)] Let $\deg\varphi=3$, then $\varphi=\varphi_{K_X}$ if and only if
 $X$ is defined by either
 $z^3=c_1^2$,or $z^3=c_2$.
\item[(7)] Let $\deg\varphi=2$, then $\varphi=\varphi_{K_X}$ if and only if
 $X$ is
 $z^2=h$.
\end{enumerate}
Here $\ell_i$'s define different lines and there are at most three
lines among them passing through one point, for $i=1\cdots 8$;
$a_i$'s are reduced of degree $2$; $b_i$'s are reduced of degree
$4$; $c_1, c_2$ are reduced of degree $6$; $h$ is reduced of degree
$8$ and all the irreducible components of $a_i$'s , $b_i$'s, $c_1,
c_2$ and $h$ are simply normal crossing.
\end{theorem}

\begin{theorem}
If $$\varphi=\varphi_{K_X}: X\longrightarrow\mathbb{P}^2,$$ is an
abelian cover, then $d$ is equal to $2,3,4,6,8,9$, or $16$. In
particular, if the canonical degree of $X$ is $36$, then
$\varphi_{K_X}$ can not be an abelian cover.
\end{theorem}

In section 2. we present the main facts on abelian covers which are the key points for solving our problem.  In section 3, we prove our main theorems. The defining equations of Tan's examples will be given in section 4.

The authors would like to thank Professor S.-L. Tan for valuable
discussions and noticing us his examples(\cite{Tan}). We are also
grateful to Professor R. Pardini for having told us some relevant
references and their paper about classifying Compedelli surfaces
with fundamental group of order $8$ (\cite{ML-P-R}).
\section{Abelain covers}
The theory of cyclic covers of algebraic surfaces was studied first
by Comessatti in \cite{Com}. Then F. Catanese (\cite{Cat2}) studied
smooth abelian covers in the case $(\mathbb{Z}_2)^{\oplus2}$ and R.
Pardini (\cite{Par2}) analyzed the general case. In this section, we
shall recall some basic definitions and results for abelian covers
which will facilitate our subsequent discussion. Since our point of
view is to find the defining equations for covering surfaces by
explicit calculation, we use the expressions and notations appearing
in \cite{Gao}.

Let $\varphi:X\to Y$ is an abelian cover associated to abelian group
$G\cong\mathbb Z_{n_1}\oplus\cdots\oplus\mathbb Z_{n_k}$, i.e.,
function field $\mathbb C(X)$ of $X$ is an abelian extension of the
rational function field $\mathbb C(Y)$ with Galois group $G$.  Without lose of
generality, we can assume $n_1|n_2\cdots|n_k$.

\begin{definition}
The dates of abelian cover over $Y$ with group $G$ are $k$ effective
divisors $D_1$, $\cdots$, $D_k$ and $k$ linear equivalent relations
\[D_1\sim n_1L_1, \cdots, D_k\sim n_kL_k.\]
\end{definition}

Let $\mathscr{L}_i=\mathscr{O}_Y(L_i)$ and $f_i$ be the defining
equation of $D_i$, i.e., $D_i=\text{div}(f_i)$, where $f_i\in H^0(Y,
\mathscr{L}_i^{n_i})$. Denote
$\textbf{V}(\mathscr{L}_i)=\textbf{Spec}S(\mathscr{L}_i)$ to be the
line bundle corresponding to $\mathscr{L}_i$, where
$S(\mathscr{L}_i)$ is the sheaf of symmetric $\mathscr{O}_Y$
algebra. Let $z_i$ be the fiber coordinate of
$\textbf{V}(\mathscr{L}_i)$. Then the abelian cover can be realized
by the normalizing of surface $V$ defined by the system of equations
\[z_1^{n_1}=f_1, \cdots, z_k^{n_k}=f_k.\]
So we have the following diagram:

\begin{diagram}
X     & \rTo^{\text{normalization}}   & V & \rdTo^f \rInto  &\oplus_{i=1}^k\textbf{V}(\mathscr{L}_i)\\
      & \rdTo(4,2)_\varphi            &   & \rdTo           &\dTo^p\\
      &                               &   &                 &Y .\\
\end{diagram}
Sometimes we call $X$ is defined by these equations if there is no
confusions in the context.

 We summerize our main results as follows.
\begin{theorem}\label{F}(See \cite{Gao}) Denote by $[Z]$ the integral part of a
$\mathbb Q$-divisor $Z$, $-L_g=-\sum\limits_{i=1}^k g_i{L}_i+
\left[\sum\limits_{i=1}^{k}\frac{g_i}{n_i}D_i\right]$.Then
\begin{eqnarray}
\varphi_*\mathcal{O}_X&=\bigoplus\limits_{{g\in G}}\mathcal
O_Y(-L_g).
\end{eqnarray}
where $g=(g_1,\cdots, g_k)\in G$.

\end{theorem}

 So the
decomposition of $\varphi_*\mathcal{O}_X$ is totally determined by
the abelian cover.

By Theorem \ref{F} we can get following corollary.
\begin{corollary}\label{h}
If $X$ is non-singular, $D$ is the divisor on $Y$, then
$$h^i(X, \varphi^*\mathcal{O}_Y(D))=\sum\limits_{g\in G}h^i(Y, \mathcal{O}_Y(D-L_g))$$
\end{corollary}

The following result will be used to calculate the ramification
index.
\begin{theorem} \label{branch}Let $P$ be an irreducible and reduced hypersurface in $Y$,
let $\bar P=\pi^{-1}(P)$ be the reduced preimage of $P$ in $X$, and
let $a_i$ be the multiplicity of $P$ in $D_i=\textrm{div}(f_i)$.
Then
$$
\pi^*P=\dfrac{|G|}{d_P}\bar P,
$$
where
$$ d_P= \gcd\left(\,|G|, \ |G|\dfrac{a_1}{n_1}
,\ \cdots, \ |G|\dfrac{a_k}{n_k}  \, \right)$$
 is the number of points in the preimage of a generic point on
 $P$.
\end{theorem}

\section{canonical map}
Let X be a surface of general type whose canonical map
$\varphi_{K_X}$ is a generically finite cover of a surface in
$\mathbb{P}^2$.

 \begin{lemma} Assume $\varphi: X\rightarrow \mathbb{P}^2$ is a
 finite cover of degree $36$,
 $\varphi_{K_X}=\varphi$ if and only if
$$\varphi_*\mathcal{O}_X=\mathcal{O}_{\mathbb{P}^2}\oplus E\oplus
\mathcal{O}_{\mathbb{P}^2}(-4),$$

 where $E$ is a rank $34$ bundle satisfying
 $E^\vee\cong E(4)$ and $H^0(E(1))=0$.
\end{lemma}

\begin{proof}If $\varphi$ is a finite cover, we take $E_0$ as the trace free part of
$\varphi_*\mathcal{O}_X$, so
$\varphi_*\mathcal{O}_X=\mathcal{O}_{\mathbb{P}^2}\oplus E_0$. By
relative duality,
$$\varphi_*\omega_X \cong (\varphi_*\mathcal{O}_X)^\vee \otimes \omega_{\mathbb{P}^2}
= \mathcal{O}_{\mathbb{P}^2}(-3)\oplus  E_0^\vee(-3).$$

Since $\omega_X=\varphi^*(\mathcal{O}_{\mathbb{P}^2}(1))$, by
projection formula,
$\varphi_*\omega_X=\varphi_*\varphi^*(\mathcal{O}_{\mathbb{P}^2}(1))=\mathcal{O}_{\mathbb{P}^2}(1)\otimes
\varphi_*\mathcal{O}_X$, so
$$\mathcal{O}_{\mathbb{P}^2}(1) \oplus  E_0(1)\cong
\mathcal{O}_{\mathbb{P}^2}(-3)\oplus  E_0^\vee(-3).$$

We see that $E_0=E\oplus \mathcal{O}_{\mathbb{P}^2}(-4)$, and
$E^\vee=E(4)$.

On the other hand, if
$\varphi_*\mathcal{O}_X=\mathcal{O}_{\mathbb{P}^2}\oplus E\oplus
\mathcal{O}_{\mathbb{P}^2}(-4)$, then
$$\begin{array}{ll}
\varphi_*(\omega_X\otimes
\varphi^*\mathcal{O}_{\mathbb{P}^2}(-1))&=\varphi_*\omega_X\otimes\mathcal{O}_{\mathbb{P}^2}(-1)\\
&=(\varphi_*\mathcal{O}_X)^\vee\otimes
\omega_{\mathbb{P}^2}\otimes\mathcal{O}_{\mathbb{P}^2}(-1)\\
&=(\mathcal{O}_{\mathbb{P}^2}\oplus E\oplus
\mathcal{O}_{\mathbb{P}^2}(-4))^\vee\otimes\mathcal{O}_{\mathbb{P}^2}(-4)\\
&=(\mathcal{O}_{\mathbb{P}^2}\oplus E^\vee\oplus
\mathcal{O}_{\mathbb{P}^2}(4))\otimes\mathcal{O}_{\mathbb{P}^2}(-4)\\
&=\mathcal{O}_{\mathbb{P}^2}(-4)\oplus E \oplus
\mathcal{O}_{\mathbb{P}^2}\\
\end{array}$$
$\mathcal{O}_{\mathbb{P}^2}(-4)\oplus E \oplus
\mathcal{O}_{\mathbb{P}^2}$ admits a non zero global section, so
$\omega_X\otimes \varphi^*\mathcal{O}_{\mathbb{P}^2}(-1)$ admits
also non zero global section. Namely there is an effective divisor
$Z$ such that $K_X=\varphi^*H+Z$, where $H$ is the divisor of the
hyperplane. Because
$$p_g(X)=h^2(\varphi_*\mathcal{O}_X)=h^2(\mathcal{O}_{\mathbb{P}^2}\oplus E\oplus\mathcal{O}_{\mathbb{P}^2}(-4))
=h^0(E(1))+3=3.$$ $Z$ is the fixed part of $|K_X|$. It implies that
$$K_X^2\geqslant (\varphi^*H)^2=36=9(p_g(X)+1)\geqslant 9\chi(\mathcal{O}_X),$$
by Miyaoka-Yau inequality, $K_X^2=36$, $Z=0$ and $q(X)=0$. So
$\varphi_{K_X}=\varphi$.
\end{proof}

\begin{lemma}\label{decom} If $\varphi=\varphi_{K_X}$
is a finite abelian cover of degree $d$ over $\mathbb{P}^2$, then
$\varphi_*\mathcal{O}_X=\mathcal{O}_{\mathbb{P}^2}\oplus\mathcal{O}_{\mathbb{P}^2}(-2)^{\oplus
d-2} \oplus\mathcal{O}_{\mathbb{P}^2}(-4)$.
\end{lemma}

\begin{proof}
Because $\varphi$ is an abelian cover, $\varphi_*\mathcal{O}_X$ is a
direct sum of the line bundle.
$$\varphi_*\mathcal{O}_X=\mathcal{O}_{\mathbb{P}^2}\oplus
\bigoplus_{i=1}^{d-1}\mathcal{O}_{\mathbb{P}^2}(-l_i).$$

 Assume
$0<l_1\leqslant l_2\leqslant \cdots \leqslant l_{d-1}$. Since
$K_X=\varphi^*(\mathcal{O}_{\mathbb{P}^2}(1))$, for any $m\geqslant
1$,
$$p_m(X)=h^0(mK_X)=h^0(\varphi^*(\mathcal{O}_{\mathbb{P}^2}(m)))
=h^0(\mathcal{O}_{\mathbb{P}^2}(m))+
\sum_{i=1}^{d-1}h^0(\mathcal{O}_{\mathbb{P}^2}(m-l_i)).$$

Because
$p_g(X)=p_1(X)=h^0(\varphi^*(\mathcal{O}_{\mathbb{P}^2}(1)))=h^0(\mathcal{O}_{\mathbb{P}^2}(1))=3$,
we see that
$$h^0(\mathcal{O}_{\mathbb{P}^2}(1-l_i))=0, \forall i $$

So $l_i\geqslant 2$.

And $h^2(\varphi_*\mathcal{O}_X)=h^2(\mathcal{O}_{\mathbb{P}^2})
+\sum_{i=1}^{d-1}h^2(\mathcal{O}_{\mathbb{P}^2}(-l_i))$,

So $3=\sum_{i=1}^{d-1}h^0(\mathcal{O}_{\mathbb{P}^2}(l_i-3))$, then
$l_i \leqslant 4$.

Then there are two cases:

\begin{itemize}
\item[(1)]$l_1=\cdots=l_{d-2}=2,\quad l_{d-1}=4$, and
\item[(2)]$l_1=\cdots=l_{d-4}=2,\quad l_{d-3}=l_{d-2}=l_{d-1}=3$.
\end{itemize}
On the other hand, if $m=2$, we have the second plurigenus
$p_2(X)=K_X^2+\chi(\mathcal{O}_X)=d+4$. So
$$d+4=h^0(\mathcal{O}_{\mathbb{P}^2}(2))+
\sum_{i=1}^{d-1}h^0\big(\mathcal{O}_{\mathbb{P}^2}(2-l_i)\big).$$

The second case does not satisfy the equation. So the lemma is
proved.
\end{proof}

Now suppose $\varphi: X\rightarrow \mathbb{P}^2$ be an abelian cover
associated to an abelian group $G\cong\mathbb
Z_{n_1}\oplus\cdots\oplus\mathbb Z_{n_k}$. Then $X$ is the
normalization of the surface defined by
\[z_1^{n_1}=f_1=\prod_\alpha p_\alpha^{\alpha_1}, \cdots, z_k^{n_k}=f_k=\prod_\alpha p_\alpha^{\alpha_k},\]
where $p_{\alpha}$'s are prime factors and $\alpha=(\alpha_1,\cdots,
\alpha_k)\in G$. Denote $x_\alpha$ to be the degree of $p_\alpha$,
$e_i=(0,\cdots,0,1,0,\cdots,0)\in G$, $1\leq i \leq k$, and $l_g$ be
the degree of $L_g$. So $x_\alpha$ and $l_g$ are all integers. Then
\begin{eqnarray}\label{s}
&n_il_{e_i}=\sum\limits_\alpha \alpha_ix_\alpha\quad i=1,\cdots k,\\
& l_g=\sum\limits_{i=1}^k g_i
l_{e_i}-\sum\limits_\alpha\left[\sum\limits_{i=1}^k
\frac{\displaystyle{g_i\alpha_i}}{\displaystyle
 {n_i}}\right]x_\alpha.
 \end{eqnarray}

 \begin{lemma}{\label{jfc}}
Using the notation as above, if $\varphi=\varphi_{K_X}$, then there
exists $g'=(g'_1,\cdots, g'_k)\in G\cong\mathbb
Z_{n_1}\oplus\cdots\oplus\mathbb Z_{n_k}$, such that $x_{\alpha}$
satisfies the following equations
$${\LARGE{(*)}}\quad\quad\left\{\begin{array}{l}
n_il_{e_i}=\sum\limits_\alpha \alpha_ix_\alpha\\
l_{g'}=\sum\limits_{i=1}^k g'_i
l_{e_i}-\sum\limits_\alpha\left[\sum\limits_{i=1}^k
\frac{\displaystyle{g'_i\alpha_i}}{\displaystyle
 {n_i}}\right]x_\alpha=4 \hspace{1cm} \\
l_g=\sum\limits_{i=1}^k g_i
l_{e_i}-\sum\limits_\alpha\left[\sum\limits_{i=1}^k
\frac{\displaystyle{g_i\alpha_i}}{\displaystyle
 {n_i}}\right]x_\alpha=2, \quad \quad g\neq g', g\in G
\end{array}\right.$$
\end{lemma}
\begin{proof} It comes from Lemma \ref{decom} directly.

\end{proof}
By the above lemma, finding surfaces whose canonical map are abelian
covers over $\mathbb{P}^2$ is equivalent to finding the integral
roots $\{x_\alpha\}$ of the above equations.

\begin{theorem}\label{36}If the canonical degree of $X$ is $36$, then $\varphi_{K_X}$ can not be an abelian cover.
\end{theorem}

\begin{proof} From the famous theorem of Beauville(\cite{Bea}), we know the image of $\varphi_{K_X}(X)$ is $\mathbb{P}^2$. We assume that $\varphi_{K_X}:X\rightarrow
\mathbb{P}^2$ is an abelian cover associated to abelian group $G$.
Then $|G|=36$.

If $G=\mathbb{Z}_{36}$, then there exists some $g'\in
\mathbb{Z}_{36}$ and corresponding $l$ such  that $\{x_\alpha\}$ must be satisfied the following
equations:
$$\left\{\begin{array}{l}
l_{g'}=g'l_1-\sum_{\alpha=1}^{35}[\frac{g'\alpha}{36}]x_\alpha=4,\\
 l_g=gl_1-\sum_{\alpha=1}^{35}[\frac{g\alpha}{36}]x_\alpha=2, \quad\quad g\neq g', \quad 1\leq
 g\leq35\\
36l_1=\sum_{\alpha=1}^{35}\alpha x_\alpha.\\ \end{array}\right.$$
After computation, we find that there is no non-negative integral
solution for the above equations for any $1\leq
 g'\leq35$.

Using the same method, we find there is no non-negative integral
solution for $(*)$ as $G=\mathbb{Z}_{n_1}\oplus\mathbb{Z}_{n_2}$ and
$n_1|n_2$.

So $\varphi_{K_X}$ can not be abelian.

\end{proof}

\begin{theorem}
Assume $\varphi: X\longrightarrow\mathbb{P}^2$ is an abelian cover.
\begin{enumerate}
\item[(1)] Let $\deg\varphi=16$, then $\varphi=\varphi_{K_X}$ if and only if $X$
is defined by
\begin{equation}
z_1^2=\ell_1\ell_3\ell_4\ell_7,\hskip0.3cm
z_2^2=\ell_1\ell_2\ell_4\ell_5,\hskip0.3cm
z_3^2=\ell_1\ell_2\ell_3\ell_6,\hskip0.3cm
z_4^2=\ell_2\ell_5\ell_6\ell_8.\label{8}
\end{equation}
In particular, the surface defined by the first three equations is a
Campedelli surface.
\item[(2)]Let $\deg\varphi=9$, then $\varphi=\varphi_{K_X}$ if and only if
$X$ is defined by $z_1^3=a_1a_2^2,z_2^3=a_1a_2a_3$.
\item[(3)] Let $\deg\varphi=8$, then $\varphi=\varphi_{K_X}$ if and only if
$X$ is defined by either $z_1^2=a_1a_4,z_2^2=a_2a_4,z_3^2=a_3a_4$,
or $z_1^2=a_1a_2,z_2^4=a_2^3a_3$.
\item[(4)] Let $\deg\varphi=6$, then
$\varphi=\varphi_{K_X}$ if and only if $X$ is defined by
$z_1^2=a_1a_2,z_2^3=a_2a_3^2$.
\item[(5)] Let $\deg\varphi=4$, then
$\varphi=\varphi_{K_X}$ if and only if
 $X$ is defined by either
 $z^4=a_1^2b_2^3$,or $z_1^2=b_1$,$z_2^2=b_2$.
\item[(6)] Let $\deg\varphi=3$, then $\varphi=\varphi_{K_X}$ if and only if
 $X$ is defined by either
 $z^3=c_1^2$,or $z^3=c_2$.
\item[(7)] Let $\deg\varphi=2$, then $\varphi=\varphi_{K_X}$ if and only if
 $X$ is
 $z^2=h$.
\end{enumerate}
Here $\ell_i$'s define different lines and there are at most three
lines among them passing through one point, for $i=1\cdots 8$;
$a_i$'s are reduced of degree $2$; $b_i$'s are reduced of degree
$4$; $c_1, c_2$ are reduced of degree $6$; $h$ is reduced of degree
$8$ and all the irreducible components of $a_i$'s , $b_i$'s, $c_1,
c_2$ and $h$ are simply normal crossing.
\end{theorem}
 \begin{proof}
 If $\varphi=\varphi_{K_X}$, then $K_X=\varphi^*(\mathcal{O}_{\mathbb{P}^2}(1))$, $|K_X|$ has no fixed part
 and is base point free.

Let $G=\mathbb{Z}_{n_1}\oplus\cdots\oplus\mathbb{Z}_{n_k}$. Since
the arguments are similar, we only prove the most complicated case
when $|G|=16$. By Lemma \ref{jfc}, we only need to find the integral
solution of the equations (*). After computation, we find if and
only if
$G=\mathbb{Z}_2\oplus\mathbb{Z}_2\oplus\mathbb{Z}_2\oplus\mathbb{Z}_2$,
there are integral solutions.
\begin{enumerate}
\item[(1)]$l_{(1,0,0,0)}=4$,
$x_{(1,1,0,1)}=x_{(1,1,1,0)}=x_{(1,0,1,1)}=x_{(1,0,1,0)}=x_{(1,0,0,1)}=x_{(1,0,0,0)}=x_{(1,1,1,1)}=x_{(1,1,0,0)}=1$,
$x_{(0,1,0,0)}=x_{(0,0,0,1)}=x_{(0,0,1,0)}=x_{(0,0,1,1)}=x_{(0,1,1,0)}=x_{(0,1,0,1)}=x_{(0,1,1,1)}=0$,
\item[(2)]$l_{(1,1,0,0)}=4$,
$x_{(0,1,1,1)}=x_{(1,0,1,1)}=x_{(0,1,1,0)}=x_{(1,0,1,0)}=x_{(1,0,0,1)}=x_{(0,1,0,0)}=x_{(1,0,0,0)}=x_{(0,1,0,1)}=1$,
$x_{(0,0,1,1)}=x_{(1,1,0,0)}=x_{(1,1,1,1)}=x_{(1,1,0,1)}=x_{(1,1,1,0)}=x_{(0,0,0,1)}=x_{(0,0,1,0)}=0$,
\item[(3)]$l_{(1,1,1,0)}=4$,
$x_{(1,1,1,0)}=x_{(0,1,0,1)}=x_{(1,0,0,1)}=x_{(0,1,0,0)}=x_{(1,0,0,0)}=x_{(1,1,1,1)}=x_{(0,0,1,0)}=x_{(0,0,1,1)}=1$,
$x_{(0,1,0,1)}=x_{(1,1,0,0)}=x_{(0,1,1,0)}=x_{(0,1,1,1)}=x_{(1,1,0,1)}=x_{(1,0,1,1)}=x_{(0,0,0,1)}=0$,
\item[(4)]$l_{(1,1,1,1)}=4$,
$x_{(0,1,1,1)}=x_{(1,1,0,1)}=x_{(1,1,1,0)}=x_{(1,0,1,1)}=x_{(0,1,0,0)}=x_{(1,0,0,0)}=x_{(0,0,1,0)}=x_{(0,0,0,1)}=1$,
$x_{(0,0,1,1)}=x_{(1,0,1,0)}=x_{(1,1,0,0)}=x_{(1,0,0,1)}=x_{(0,1,1,0)}=x_{(0,1,0,1)}=x_{(1,1,1,1)}=0$.

\end{enumerate}

 Although there are 4 sets of
solution, these defining equations are equivalent after quotient the
group $sl(4,
 \mathbb{Z}_2)$ which means that they are same surfaces. So the surface
 $X$
 is defined by \eqref{8}.
Next we are going to discuss the configurations of $\ell_i$'s as
following three cases.
\begin{itemize}
\item  If $\ell_i$'s are in general position, i.e. they are simply
normal crossing. Now we prove $X$ is smooth.

Let $p_{ij}$ be the intersection point of $\ell_i$ and $\ell_j$. The
cover is locally defined by \[z_1^2=x^{a_{11}}y^{a_{12}}, \quad
z_2^2=x^{a_{21}}y^{a_{22}}, \quad z_3^2=x^{a_{31}}y^{a_{32}},\quad
z_4^2=x^{a_{41}}y^{a_{42}},\] where $a_{ij}=0$ or $1$ for all $i$,
$j$.

It is easy to check that $\{(a_{i1}, a_{i2})\}\nsubseteqq
\{(1,1),(0,0)\}$ i.e., at least one pair $\{(a_{i1},
a_{i2})\}=\{(1,0)\}$ or $\{(0, 1)\}.$  Without lose of generality,
we can assume $(a_{11},a_{21})=(1,0)$, i.e. $z_1^2=x$. The cover is
branched along the smooth line. So the surface is smooth.

From Theorem \ref{branch}, the ramification index of
$H_i=\Div(\ell_i)$
 is $2$.
Thus
$K_X=\varphi^*(\mathcal{O}_{\mathbb{P}^2}(-3)+\frac{1}{2}\sum_{i=1}^{8}H_i)
=\varphi^*(\mathcal{O}_{\mathbb{P}^2}(1))$, which means
$\varphi=\varphi_{K_X}$.
\vspace{.5cm}
\item  If there are three of $\ell_i$'s intersecting at a point, we blow up
these triple points at first. Let $\sigma: P\rightarrow
\mathbb{P}^2$ be the blowing-ups of these triple points with
$\left\{E_s\right\}$ the except curves and $\pi: \Sigma\rightarrow
P$ be the corresponding abelian cover, i.e., the pull back of
$\varphi$ by $\sigma$. Then it is easy to check that $E_s$'s are in
the branch locus of $\pi$ and the ramification indices are also $2$.
Similarly, we can show that $\Sigma$ is smooth. Then the canonical
divosor
\begin{eqnarray*}
K_\Sigma&=&\pi^*(\sigma^*(-3H)+\sum_{s}E_s+\frac{1}{2}(\sigma^*(8H)-3\sum_{s}E_s)+\frac{1}{2}\sum_{s}E_s)\\
&=&\pi^*\sigma^*(H),
\end{eqnarray*}
where $H$ is a hyperplane on $\mathbb{P}^2$. So the surface is
minimal and the degree of the canonical map is $16$. In this case,
$X$ has only $A_1$ type singularities. \vspace{.5cm}
\item  If there are $d$ lines among $\ell_i$'s passing through a point,
$d\geq 4$. We blow up these points such that the branch locus are
normal crossing. Let $\sigma: P\rightarrow \mathbb{P}^2$ be the
blowing-ups of these points with $\left\{E_s| s \in S\right\}$ the
except curves and $\pi: \Sigma\rightarrow P$ be the corresponding
abelian cover. Then $\Sigma$ is smooth. $\left\{E_s|s \in S\right\}$
decomposes into $2$ parts: $\left\{E_s|s \in S_1\right\}$ are in the
branch locus of $\pi$ and $\left\{E_s| s \in S_2\right\}$ are not .
Then the canonical divisor
\begin{eqnarray*}
K_\Sigma
&=&\pi^*(\sigma^*(-3H)+\sum_{s\in S}E_s+\frac{1}{2}(\sigma^*(8H)-d\sum_{s\in S}E_s)+\frac{1}{2}\sum_{s\in S_1}E_s)\\
&=&\pi^*(\sigma^*(H)+\frac{3-d}{2}\sum_{s\in
S_1}E_s+\frac{2-d}{2}\sum_{s\in S_2}E_s),
\end{eqnarray*}
where $H$ is a hyperplane on $\mathbb{P}^2$. Obviously,
$K_\Sigma^2<16$. There is no $-1$ curves on $\Sigma$ by resolution
and the singularities on $X$ are not rational double points.
\end{itemize}

Therefore the configuration of $\ell_i$ has at most triple points.

Let $X_1$ be the surface defined by the first three equations of
\eqref{8} and suppose $\pi': X_1\rightarrow \mathbb{P}^2$. When
$\ell_i$'s are in general position, we can show that $X_1$ is smooth
and $K_{X_1}^2=2$. By Theorem \ref{F}, $L_g\sim 2H$ for $g \neq 0$,
where $H$ is a hyperplane on $\mathbb{P}^2$. So
$h^0(-L_g)=h^0(K_{\mathbb{P}^2}+L_g)=0$ for all nonzero $g \in
\mathbb{Z}_2\oplus \mathbb{Z}_2\oplus\mathbb{Z}_2$. Due to Corollary
\ref{h}, $p_g(X_1)=q(X_1)=0$. Therefore the surface $X_1$ is a
Campedelli surface.

Now we compute the fundamental group of the surface $X_1$.

According to the result of the canonical resolution of double cover,
the preimage of each $L_i=\Div(\ell_i)$ by $\pi'$ is irreducible.
Then assume $\pi'^*(L_i)=2A_i$, we can get $A_i^2=2$.

Due to $2A_i\sim \pi'^*(H)$, $2(A_i-A_j)\sim 0$. Assume $A_i\sim
A_j$, $i \neq j$,
 then $h^0(A_i)\leq
 2$ since $h^0(A_i+A_j)\geq h^0(A_i)+h^0(A_j)-1$.  And the equation holds if and only if $A_i$ and $A_j$
 do not intersect, which is contradict to $A_i^2=2$.  So
 $h^0(A_i)=1$, which means that $A_i$ is not linear equivalent to $A_j$.

Therefore each $A_i-A_j$ is a $2$-torsion element in $X_1$. So $X_1$
has at least $6$ $2$-torsion elements.

Denote $B=\sum_{i=1}^7L_i$ to be the branch locus in $\mathbb{P}^2$.
$\pi_1(\mathbb{P}^2-B)=\mathbb{Z}^{\bigoplus 6}$ is an abelian
group. We know that $X_1-\pi'^*(B)$ is an $\acute{\text{e}}$tale
cover over $\mathbb{P}^2-B$. So $\pi_1(X_1-\pi'^*(B))$ is a subgroup
of $\pi_1(\mathbb{P}^2-B)$ which is abelian and $\pi_1(X_1)$ is the
quotient group of $\pi_1(X_1-\pi'^*(B))$, then $\pi_1(X_1)$ is also
an abelian group, which means $\pi_1(X_1)=H_1(X_1,\mathbb{Z})$.

As $q=0$, $H^1(X_1,\mathbb{C})=0$. Using the universal coefficient
theorem for cohomology, we know
$H^1(X_1,\mathbb{C})=H_1(X_1,\mathbb{C})=0$. Since
$H_1(X_1,\mathbb{C})=H^1(X_1,\mathbb{Z})\bigotimes \mathbb{C}=0$,
$H^1(X_1,\mathbb{Z})$ only has torsion parts.

Due to the following exact sequence
$$H^1(X_1, \mathcal{O}_{X_1})\longrightarrow H^1(X_1, \mathcal{O}^*_{X_1})\longrightarrow H^2(X_1,\mathbb{Z})\longrightarrow H^2(X_1, \mathcal{O}_{X_1})$$
and $p_g(X_1)=q(X_1)=0$, we know $\mathrm{Pic} X_1=H^1(X_1,
\mathcal{O}^*_{X_1})= H^2(X_1,\mathbb{Z})$.

From the following sequence
$$0\longrightarrow \mathrm{Ext}(H_1(X_1,\mathbb{Z}),\mathbb{Z}) \longrightarrow  H^2(X_1, \mathbb{Z})
\longrightarrow \mathrm{Hom}(H_2(X_1, \mathbb{Z}), \mathbb{Z})
\longrightarrow 0,$$ we have
$H_1(X_1,\mathbb{Z})=\mathrm{Tor}(H^2(X_1,\mathbb{Z}))$.

For any Campedelli surface $X$, it is well known that
$\big|\pi_1(X)\big|\leqslant9$(\cite{Rei}). So the torsion group of
the Picard group of $X_1$ is
$G=\mathbb{Z}_2\oplus\mathbb{Z}_2\oplus\mathbb{Z}_2$, i.e.,
$\pi_1(X_1)=\mathbb{Z}_2\oplus\mathbb{Z}_2\oplus\mathbb{Z}_2$.

\end{proof}

\begin{remark}
\begin{enumerate}
\item [(1)] From the above theorem, we know the defining equations of
Persson's surface.
\item [(2)] The Campedelli surface constructed in the proof is as same as
that given by Kulikov (\cite{Kul}). In fact, Campedelli surfaces
with fundamental group of order $8$ are completely classified in
\cite{ML-P-R}.
\end{enumerate}
\end{remark}

\begin{corollary}
If $$\varphi=\varphi_{K_X}: X\longrightarrow\mathbb{P}^2,$$ is an
abelian cover, then $d$ is equal to $2,3,4,6,8,9$ or $16$.
\end{corollary}

\section{Defining equations of Tan's Examples}\label{Examples}
In \cite{Tan}, Tan constructed a series of nice surfaces whose
canonical maps are of odd degrees by using a so-called conception
$\mathbb{Z}_p$-set of the singularities.

Let $N$ be a set of singular points of surface $\Sigma$, and let
$\sigma: Y\rightarrow \Sigma$ be the minimal resolution of $\Sigma$
with $\Gamma_1$, $\cdots$, $\Gamma_n$ the components of
$\sigma^{-1}(N)$. If there are positive integers $a_i<p$ and a
divisor $D$ on $Y$ such that $\sum_{i=1}^na_i\Gamma\sim pD$ then $N$
is called a $\mathbb{Z}_p$-set of singularities on $\Sigma$. A $p$
fold cyclic cover can be determined by the $\mathbb{Z}_p$-set.

 Take $3$ points $q_1$, $q_2$, $q_3$ in general position on $\mathbb P^2$,
and every $3$ lines passing through each point. This configuration
of $9$ lines has $n+3$ triple points and $27-3n$ double points. And
he assumed that all lines passing through $q_1$, $q_2$ contain no
$4$ triple points. Blowing up the $n+3$ triple points, we suppose
each $A_i$ to be  the strict transforms of the $3$ lines passing
through $q_i$. Then a Galois triple cover $\pi_1$ with branch locus
$A_1+A_2+A_3$ can be defined and
$N=\pi_1^{-1}(\bigcup_{i<j}A_i\bigcap A_j)$ is a $\mathbb{Z}_3$-set.
His example $X$ is the 3 fold cyclic cover determined by $N$. In
order to show the canonical map of $X$ is of degree 3, Tan
calculated the invariants of such surface. But the computation of
the invariants of the surface is pretty complicated. The main part
of \cite{Tan} is the computation of $p_g(X)$.

We can reconstruct the Tan's example by explicitly equations. Let
$\varphi: X \rightarrow \mathbb{P}^2$ be an abelian cover branched
over the line configuration described by Tan. Assume $\ell_1$,
$\ell_2$, $\ell_3$ are the $3$ lines through $q_1$, $\ell_4$,
$\ell_5$, $\ell_6$ are the $3$ lines through $q_2$ and $\ell_7$,
$\ell_8$, $\ell_9$ are the $3$ lines through $q_3$. $\varphi$ is
defined by the following equations.
$$\pi:  \hskip0.3cm  z_1^3=\ell_1\ell_2\ell_3\ell_4\ell_5\ell_6\ell_7\ell_8\ell_9,\hskip0.3cm
z_2^3=\ell_1\ell_2\ell_3\ell_4^2\ell_5^2\ell_6^2.$$

Let $\sigma: P\rightarrow \mathbb{P}^2$ be blowing-ups of
$q_1$,$q_2$,$q_3$. $\pi: X\rightarrow P$ is the pull back of
$\varphi$ by $\sigma$. Let $\Sigma$ be the minimal resolution of
$\pi':Y \rightarrow \mathbb{P}^2$ which is a cyclic cover defined by
$z^3=\ell_1\ell_2\ell_3\ell_4\ell_5\ell_6\ell_7\ell_8\ell_9$. It is
clear that $X$ and $Y$ have at most $A_1$-type singularities. We
calculate invariants $q(\Sigma)= q(X)=0$ and
$p_g(X)=p_g(\Sigma)=8-n$ by Theorem \ref{F}. The canonical map of
$X$ factorizes through that of $\Sigma$, so it is of degree $3$.

In \cite{Tan}, the author also constructed a surface whose canonical
map is of degree 5 by using $\mathbb{Z}_5$ set. We can also
reconstruct it as follows.

Take five lines $\ell_1$, $\cdots$, $\ell_5$ in general position on
$\mathbb{P}^2$. $\pi: X\rightarrow \mathbb{P}^2$ is an abelian cover
defined by
$$z_1^5=\ell_1\ell_2\ell_3\ell_4\ell_5,\hskip 1cm z_2^5=\ell_1\ell_2^2\ell_3^3\ell_4^4. $$

The 5 fold covering surface $\Sigma$ over $\mathbb{P}^2$ is defined
by the first equation. Similarly, we get $K_X=25$, $K_{\Sigma}=5$,
$\chi(\mathcal{O}_X)=\chi(\mathcal{O}_{\Sigma})=5$,
$p_g(\mathcal{O}_X)=p_g(\mathcal{O}_{\Sigma})$. Therefore, the
canonical map of $X$ factorizes through that of $\Sigma$. This means
that the degree of the canonical map of $X$ is $5$.


\begin{thebibliography}{EVW}


\bibitem [Bea]{Bea} Beauville, A.: \textit{L'application canonique pour les surfaces de type g$\acute{e}$n$\acute{e}$ral},
Invention Math. \textbf{55}, 121-140(1979).

\bibitem [Cat1]{Cat1} Catanese, F.: \textit{Babbage's conjecture, contact of surfaces, symmetric determinatal varieties and
applications}, Invent. Math. \textbf{63} (1981), 433-465.

\bibitem [Cat2]{Cat2} Catanese, F.: \textit{On the moduli spaces of surfaces of general type}, J.
Differential Geom., \textbf{19} (1984), 483-515.

\bibitem [Com]{Com} Comessatti, A.: \textit{Sulle superfici multiple cicliche}, Rend. Sem. Mat.
Univ. Padova, \textbf{1} (1930), 1-45.

\bibitem [Gao]{Gao} Gao, Y.: \textit{A note on finite abelian
cover}, Science in China, Series A: Mathematics \textbf{54} (2011),
1333-1342.

\bibitem [Kul]{Kul} Kulikov, Vik. S.: \textit{Old examples and a new example of surfaces of general type with $p_g=0$}, (Russian) Izv. Ross. Akad. Nauk Ser. Mat. \textbf{68} (2004), no. 5, 123-170; translation in Izv. Math. \text{68} (2004), no. 5,
965-1008.

\bibitem [ML-P-R]{ML-P-R} Mendes Lopes, M., Pardini, R. and Reid, M.:\textit{Campedelli surfaces with fundamental group of order 8}, Geom. Dedicata \textbf{139} (2009), 49-55.

\bibitem [Par1]{Par1} Pardini, R.: \textit{Canonical images of
surfaces}, J. reine angew. Math. \textbf{417}(1991), 215-219.

\bibitem [Par2]{Par2} Pardini, R.: \textit{Abelian covers of algebraic varieties}, J. Reine Angew. Math.
\textbf{417}(1991), 191¨C213.

\bibitem [Per]{Per} Persson, U.: \textit{Double coverings and surfaces of general type.}, In: Olson,L.D.(ed.) Algebraic geometry.(Lect. Notes Math., vol.732, pp.168-175)
Berlin Heidelberg New York: Springer 1978.

\bibitem [Pet]{Pet} Peters, C.A.M.: \textit{On two types of surfaces of general type with vanishing geometric genus}, Inventiones Math.
\textbf{417} (1991),191-213.

\bibitem [Rei]{Rei} Reid, M. :  \textit{Surfaces with $p_g=0, K^2=2$.} Preprint available at
http://www.warwick.ac.uk/~masda/surf/.

\bibitem [Tan]{Tan}  Tan, S.-L.: \textit{Surfaces whose canonical maps are of odd degrees}.
Math. Ann. \textbf{292} (1992), no. 1, 13--29.

\bibitem [V-Z]{V-Z} van Der Geer, G., Zagier, D:  \textit{The Hilbert modular group for the field
$\mathbb{Q}[\sqrt{13}]$}, Invent. Math. \textbf{42} (1977), 93-134.

\bibitem [Xiao]{Xiao} Xiao, G.: \textit{Algebraic surfaces with high canonical
degree}, Math. Ann.,\textbf{274}(1986), 473-483.






\end{thebibliography}
\end{document}